\documentclass[12pt,twoside]{article}
\usepackage{a4wide}
\usepackage{amssymb,amsfonts,amsthm,amsmath}
\usepackage[utf8]{inputenc}
\usepackage[english]{babel}
\usepackage{epsfig}
\usepackage{hyperref}
\usepackage{array}
\usepackage{graphicx,epstopdf}
\usepackage{latexsym,color, enumerate}
\usepackage{bbm,url}
\usepackage{setspace}
\usepackage{caption,subcaption}

\usepackage{mathalpha,mathrsfs}

\clearpage
\vspace{3cm}

\newtheorem{theorem}{Theorem}[]
\newtheorem{lemma}[theorem]{\bf Lemma}

\newtheorem{problem}[theorem]{\bf Problem}
\newtheorem{conjecture}[theorem]{\bf Conjecture}

\newtheorem{claim}[theorem]{\bf Claim}

\newtheorem{definition}[theorem]{\bf Definition}

\begin{document}
\newcommand\cA{\mathcal{A}}
\newcommand\cB{\mathcal{B}}
\newcommand\cC{\mathcal{C}}
\newcommand\cF{\mathcal{F}}
\newcommand\cG{\mathcal{G}}
\newcommand\cH{\mathcal{H}}
\newcommand\cJ{\mathcal{J}}
\newcommand\cL{\mathcal{L}}
\newcommand\cM{\mathcal{M}}
\newcommand\cQ{\mathcal{Q}}
\newcommand\cR{\mathcal{R}}
\newcommand\cS{\mathcal{S}}
\newcommand{\cSk}{{\mathcal S}_k}
\newcommand{\cSd}{{\mathcal S}_d}
\newcommand{\cSt}{{\mathcal S}_3}
\newcommand\EE{\text{\rm E}}
\newcommand\ex{{\rm ex}}

\newcommand{\todo}[1]{\textcolor{red}{TODO: {#1}}}

\def \today {\ifcase \month \or January\or February\or March\or April\or
  May\or June\or July\or August\or September\or October\or November\or
  December\fi\ \number \day, \number \year}

\pagestyle{myheadings}
\markboth{{\small \sc F\"uredi, Keszegh, Manuel:}}{{\it\small Forbidden stars and visibility of lattice points \hfill {\rm \today}${}$\quad${}$}}

\title{Forbidden stars in multidimensional $0$-$1$ matrices\\ and visibility of lattice points}

\author{Zolt\'an F\"uredi\thanks{HUN-REN Alfréd Rényi Institute of Mathematics, Budapest, Hungary}\and 
Bal\'azs Keszegh\thanks{HUN-REN Alfréd Rényi Institute of Mathematics and ELTE Eötvös Loránd University, Budapest, Hungary. Supported by the ERC Advanced Grant ``ERMiD'', no.~101054936 and by the EXCELLENCE-24 project no.~151504 Combinatorics and Geometry of the NRDI Fund.}\and 
Paul Manuel\thanks{Department of Information Science, College of Life Sciences, Kuwait University, Kuwait}}

\date{}

\maketitle
	
\begin{abstract}
 A $d$-dimensional $0$-$1$ matrix $M$ of size $n_1\times n_2\times \dots \times n_d$ can be considered as a Boolean function
  $M: B(n_1\times n_2\times \dots \times n_d) \to \{ 0,1\}$, where $B$ is the $d$-dimensional box of lattice points $(x_1, \dots, , x_d)\in Z^d$ with  $0\leq x_i \leq n_i-1$, $1\leq i\leq d$. The $0$-$1$ matrix $M$ can also be described as a subset $P:=P(M)$ of $B$ such that $x\in P$ if and only if $M(x)=1$. 
A $k$-star with center $p$ in $M$ corresponds to a $(k+1)$-element subset $\{ p, p_1, \dots , p_k\} \subset B$ such that
$p$ and $p_i$ differ only in one coordinate (for all $1\leq i\leq k$) and these $k$ coordinates are distinct.

Here we consider the problem of determining the maximum number of $1$-entries of a $0$-$1$ matrix $M$ of dimension $d$ and size $n\times n \times \dots \times n$ that avoids all $k$-stars.
 Our main results are the asymptotical solution of the problem for every $d$ and $k$ (as $n\to \infty$), very close bounds for $k=d$, and the exact solution of the $d=k=3$ case.

This problem has connections to several other areas of discrete mathematics, including $k$-partite hypergraphs, independent set problems, dominating set problems and covering codes.
One of our tools (concerning maximal packings of induced copies of a given hypergraph, Theorem~\ref{thm:induced_new}) might have independent interest.

\end{abstract}

\section{Introduction}
\subsection{A short history of matrix pattern avoidance problems}

A central family of questions of combinatorics are {\em Tur\'an type problems}, when one wants to find the largest structure that doesn't include a given forbidden substructure.
Paul Tur\'an introduced this problem in graph theory~\cite{Turan-41}.
Tur\'an type problems of $0$-$1$-matrices (also called \emph{Boolean matrices}) are  widely investigated in part due to their applications in Information Theory and Computer Graphics~\cite{Edwards-72}. Two dimensional $0$-$1$ matrices can be regarded as incidence matrices of bipartite graphs with a linear order on their vertex parts. With this analog in mind, we say that a $d$-dimensional $0$-$1$ matrix $M$ \emph{contains} another $0$-$1$ matrix $A$ if $M$ has a submatrix that can be transformed to $A$ by changing any number of $1$-entries to $0$-entries. Otherwise, $M$ \emph{avoids} $A$. Note that the rows need not be consecutive in a submatrix, but the order of rows has to be preserved. The matrix $A$ is usually referred to as the \emph{pattern} to be avoided and the extremal function of $A$ is the maximum number of $1$-entries in an $n\times n$ matrix avoiding $A$. These notions generalize naturally to families of patterns that we want to avoid.

A lot of research concentrates on the two dimensional case, avoiding one or a few patterns. To stay brief, we mention the celebrated result of Marcus and Tardos~\cite{MaTa-04} showing that for any permutation pattern, its extremal function is linear in $n$.
Pattern avoidance problems have strong connections to Davenport-Schinzel theory. Some interesting applications of the pattern avoidance problem of $0$-$1$ matrix theory are found in the construction of algorithms and data structures as observed by Pettie~\cite{Pettie-10}. One interesting sequence of research in pattern avoidance was originated by the three conjectures of F\"uredi and Hajnal~\cite{FuHa-91}. For more chronological history see, e.g.,~\cite{Geneson-21}.

When considering pattern avoidance of higher dimensional $0$-$1$ matrices, a very important example is the famous density Hales-Jewett theorem \cite{HJ}. For specific recent results see~\cite{GeTi-15, GeTs-23}.

Observe that when we avoid the family of every ordering of the rows/columns of the incidence graph of a specific bipartite graph, then the $2$-dimensional matrix family avoidance problem becomes equivalent to the (unordered) Tur\'an problem of avoiding that graph as a subgraph in a bipartite graph. The same way, higher dimensional matrix avoidance problems can emulate subhypergraph avoidance problems. This is also the case for the star avoidance problem that we will investigate in this paper, and this connection will be heavily used in the paper, see also Section \ref{sec:further} about history of this area.

For further related problems, especially for which our results about star avoidance have implications, see Section \ref{sec:further}.

\subsection{Definitions, Stars in multidimensional matrices}
We use some of the notation of~\cite{GeTi-15, GeTs-23}.  A $d$-dimensional $0$-$1$ matrix $M$ of size $n_1\times n_2\times \dots \times n_d$ can be considered as a Boolean function
  $M: B(n_1\times n_2\times \dots \times n_d) \to \{ 0,1\}$, where $B$ is a $d$-dimensional {\em box of lattice points} $(x_1, \dots, , x_d)\in Z^d$ with  $0\leq x_i \leq n_i-1$, $1\leq i\leq d$. The $0$-$1$ matrix $M$ can also be described as a subset $P:=P(M)$ of $B$ such that $x\in P$ if and only if $M(x)=1$.
  The number of $1$-entries of $M$ is called the \emph{weight} of $M$. A \emph{$j$-row} $r_M(x,j)$ of a matrix $M$
  through the element $x\in B$ consists of the entries of $M$ whose coordinates are fixed except for the $j$'th coordinate, i.e.,
  it is a $1$-dimensional submatrix $$r_M(x,j):=\{ M(y): y\in B, \, x \text{ and } y \text{ differ only in the } j\text{'th coordinate}\}.$$ 
More generally, a $k$ dimensional {\em cross section} 
  is the set of all entries whose coordinates are arbitrary on some set of $k$ coordinates 
  and fixed on the other coordinates.  
Thus a row is a $1$ dimensional cross section and its weight is the number of $1$ entries on an axis parallel line.
A row with exactly one $1$-entry is called a \emph{single-entry line}.

When one regards a $d$ dimensional $0$-$1$ matrix $M$ geometrically as a subset $P$ of a $d$ dimensional rectangular box of lattice points (or grid points) then a $k$ dimensional cross section is the set of grid points that lie on a $k$ dimensional axis-parallel affine subspace.

Let us denote for a family of matrices $\cal A$ by $\ex_d(\cA,n_1,\ldots,n_d)$ the maximal possible weight (that is, the maximal possible number of $1$-entries) of a $d$ dimensional $n_1\times \dots \times n_d$ matrix avoiding every member of $\cal A$. We additionally denote $\ex_d(A,n_1,\ldots,n_d)=\ex_d(\{A\},n_1,\ldots,n_d)$ for a single matrix $A$. When $n_1=\ldots = n_d=n$, $\ex_d(\cA,n_1,\ldots,n_d)$ is denoted by $\ex_d(\cA,n)$.

Given a $d$ dimensional $0$-$1$ matrix $M$ a $1$-entry $p$ of $M$ is called a \emph{$k$-hub} if there are $k$ coordinates such that for each of these coordinates there exists a $1$-entry in $M$ which differs from $p$ only in this coordinate. A \emph{$k$-star} in $M$ is a $k$-hub together with these $k$ $1$-entries.
As a submatrix, a $k$-star is a matrix with $k$ of its sizes being $2$, the rest being $1$ and that has $(k+1)$ $1$-entries: one $1$-entry being the hub and the additional $k$ $1$-entries differ from this in exactly one coordinate. Let us call the family of all $d$ dimensional $k$-stars by $\cSk$.

Note that $\cSk$ for $k\ge 2$ consists of $2^k\binom{d}{k}$ $0$-$1$ matrices, it includes all $2^k$ orderings of each of the $\binom{d}{k}$ stars, because $\{0,1\}^d$ contains exactly this many $k$-stars.

\subsection{Preliminary results of the main problem of this paper}

We are interested in bounding $\ex_d(\cSk,n)$.
Before presenting our main results in the next Section we collect a few known or simple cases.  Short proofs are also included for the sake of completeness.

Observe that the function $\ex_d(\cS_k,n_1, \dots, , n_d)$ is monotone in all of its parameters.
If the matrix $M$ is $\cS_k$-free then so does the one obtained by permuting its layers.
Also one can exchange any pair of coordinates.
Since $\ex_d(\cS_k,n_1, \dots, , n_d)= \ex_{d+1}(\cS_k,n_1, \dots , n_d, 1)$,  to avoid trivialities we suppose that $n_i\geq 2$ for each
$1\leq i\leq d$ and $1\leq k\leq d$.

\paragraph{Multidimensional permutation matrices}
A $d$-dimensional $0$-$1$ matrix $M$ of size $n_1\times n_2\times \dots \times n_d$ is $\cS_1$-free if each row (in all $d$ directions) contains at most one non-zero element. These are frequently called generalized permutation matrices (see, e.g.,~\cite{BrCa-21, KlRu-95, Geneson-21}, and Linial and Simkin~\cite{Linial}).
If $n_1\geq n_2\geq \dots \geq n_d$ then
\begin{equation}\label{claim:onehub}
	\ex_d(\cS_1,n_1, \dots,  n_d)= n_2\times \dots \times n_d.  \,\text{ In particular }\,  \ex_d(\cS_1,n)=n^{d-1}.
\end{equation}
Indeed, each of the $n_2\times \dots \times n_d$ $1$-rows parallel to the first axis contains at most one $1$-entry, so the weight of $M$ is at most the number of these rows. On the other hand (as it was observed among others by Sabidussi, see~\cite{HaIm-11}) the whole box
 $B(n_1\times n_2\times \dots \times n_d)$ can be decomposed into $n_1$ permutation matrices as follows:
 $B=\cup P_\alpha$ where $0\leq \alpha\leq n_1-1$ and $P_\alpha$ consists of all $(x_1,\dots , x_d)$ with $\sum x_i= \alpha\mod n_1$.
\qed

\paragraph{The $2$-dimensional case}
(Folklore, see, e.g.,~\cite{MaBr-23}).
\begin{equation}\label{claim:2-hub}
	\ex_2(\cS_2,n_1, n_2)=n_1+n_2-2\text{ if }n_1,n_2\ge 2.\quad
 \,\text{ In particular }\,  \ex_2(\cS_2,n)=2n-2.
\end{equation}
Indeed, define $M$ such that its $1$-entries are those entries that have exactly one $0$-coordinate. It is $\cS_2$-free with weight $n_1+n_2-2$.
Now consider an $\cS_2$-free  $M$ of size $n_1\times n_2$.
To get an upper bound on its weight we proceed by induction. 
The case $n_1=n_2=2$ is obvious.
Suppose $n_1\geq n_2\geq 2$.
If there is a column (vertical row) of weight at most $1$ then we can delete that column and apply induction for the rest.
Otherwise, the weight of each column is at least two, so these $1$-entries must be in different horizontal rows.
Then the weight of $M$ is at most $n_2$.
\qed

\paragraph{Excluding a single star, an ordered version}
Here we consider the case when out of the possible $2^k\binom{d}{k}$ types of stars only one is excluded.

The $d$-dimensional $k$-star $S_k^+$ is defined as a $(k+1)$-element subset $\{ p, p_1, \dots , p_k\} \subset B$ such that
$p$ and $p_i$ differ only in the coordinate $i$ (for all $1\leq i\leq k$), and the center $p$ dominates the other members (i.e., each of its coordinates is at least as large as the corresponding coordinate of the other members).
So in this case the ordering of the matrix does matter. Determining the maximum weight of an
$S_k^+$-free matrix is easy:
$$
\ex_d(S_k^+,n_1, \dots,  n_d)= n_1\times \dots \times n_d- \prod_{i\leq k}(n_i-1)\times \prod_{k<i\leq d}n_i.$$
In particular
\begin{equation}\label{claim:onestar} \ex_d(S_k^+,n)=n^{d-k}(n^k-(n-1)^k)=kn ^{d-1}+O(k^2n^{d-2}).
\end{equation}
Indeed, define $M$ such that its $1$-entries are those entries that have at least one $0$ coordinate in the first $k$ coordinates. It is $S_k^+$-free with a weight described in~\eqref{claim:onestar}.
Now consider an $S_k^+$-free $M$.
To get an upper bound on its weight it is enough to consider each $k$ dimensional cross section of $M$ with
fixed values in the last $d-k$ positions.
	Thus we can assume that $d=k$ and it is enough to prove that $\ex_k(S_k^+)=\left(\prod_{i\leq k}n_i\right)-\prod_{i\leq k}(n_i-1)$.
To see this, take an $M$ with maximum weight. Observe that for each $1$-entry in $M$ with no $0$-coordinates, we can change one of its coordinates to $0$ to get a $0$-entry. Note that no $0$-entry can be assigned to two $1$-entries this way. Replace each such $1$-entry with this corresponding entry to get a new matrix $M'$ with the same weight, and having no $1$-entries with no $0$-coordinates. This has weight at most as required, finishing the proof.

\paragraph{A general upper bound}
Call a row $r$ a $1$-{\em entry line} if it contains exactly one element of $P$, i.e., it has  one 1-entry and all others are $0$'s.
If the matrix contains no $k$-star then for each $1$-entry $p$ there are (at least) $d-k+1$ rows that are empty except for $p$.
 As we have $dn^{d-1}$ rows (in all directions altogether) and a row corresponds to at most one $1$-entry this way, we have at most $dn^{d-1}/(d-k+1)$ $1$-entries.
\begin{equation}\label{claim:k-hub-upper}
  \ex_d(\cS_k,n)\le \frac{d}{d-k+1}n^{d-1}.
 \end{equation}

	One can get a slightly better upper bound by double counting the $0$-entries. Suppose that $M$ contains $m$  $1$'s.
The number of $1$-entry lines is at least $m(d-k+1)$. They cover (with multiplicities) at least $m(d-k+1)(n-1)$ $0$-entries.
Each $0$-entry is counted at most $d$ times (for each axis direction at most once). Hence
	$m(d-k+1)(n-1)\le d(n^d-m)$.
Reordering yields
\begin{equation*}\label{claim:k-hub-upper-better}
 \ex_d(\cS_k,n)=m \le \frac{d}{d-k+1}n^{d-1}\times \frac{n(d-k+1)}{n(d-k+1) +(k-1)}.
 \end{equation*}
In the case $k=d$ this gives
\begin{equation}\label{claim:k-hub-upper-d}
 \ex_d(\cS_d,n)\le {d}n^{d-1}\times \frac{n}{n +d-1}.
 \end{equation}

\begin{claim}\label{claim:d-hub}{\rm [The $k=d$ case]} For all $d\ge 3$ as $n\to \infty$
\begin{equation}\label{d-estimate}
d(n-1)^{d-1}+3\binom{d}{3}(n-2)^{d-3}\le  \ex_d(\cSd,n)\le  d(n-1)^{d-1}+d\binom{d}{2}(n-2)^{d-3}+O(d^4n^{d-4}).
\end{equation}
\end{claim}

The upper bound follows from~\eqref{claim:k-hub-upper-d} after some high school algebra. 
For the lower bound a simple $\cSd$-free construction of weight $d(n-1)^{d-1}$ can be obtained by taking as $1$-entries all entries of $B_d(n)$ having exactly one
 $0$ coordinate. This is not maximal, one can add $3\binom{d}{3}(n-2)^{d-3}$ the entries having two $0$ coordinates, one  $1$ coordinate, and
 $(d-2)$ coordinates with values exceeding $1$.
The very same idea of having exactly two $0$ coordinates implies the lower bound recurrence
\begin{equation}\label{d-recursion}
 \ex_d(\cS_d,n)\geq d(n-1)^{d-1} +\binom{d}{2}\ex_{d-2}(\cS_{d-2},n-1).
 \end{equation}

\section{Main results}


We have $\ex_d(\cS_k,n)=\Theta_k(n^{d-1})$ since $n^{d-1}\leq  \ex_d(\cS_k,n)$ by~\eqref{claim:onehub} and it is at most $kn ^{d-1}+O(k^2n^{d-2})$ by~\eqref{claim:onestar}.
Here we determine $\lim_{n\to\infty}\ex_d(\cS_k,n)/n^{d-1}$ (as $k,d$ are fixed) and present some exact values as well.

Using~\eqref{claim:k-hub-upper} we already have explicit asymptotically matching constructions for $k=1$ from~\eqref{claim:onehub}
 and for $k=d$ from~\eqref{d-estimate}.
One of our main results is that~\eqref{claim:k-hub-upper} is in fact the right bound.

\begin{theorem}\label{thm:k-hub}
$
\ex_d(\cSk,n)=  \frac{d}{d-k+1} n^{d-1}+o(n^{d-1})$ if $1\leq k\le d$ are fixed and $n\to \infty$.
\end{theorem}

For $1<k<d$ we present a rather involved probabilistic argument in Sections~\ref{sec:hypergraphs} and~\ref{sec:general}.
When $d/(d-k+1)$ is an integer one can make an explicit construction, as follows.
Partition the set $\{1,2,\dots ,d\}$  into  $(d-k+1)$-element blocks $I_1, \dots ,I_{d/(d-k+1)}$.
Put an element $x\in B$ into $P$ if and only if there exists one part $I_u$ in which the coordinates of $x$ sum up to $0$ modulo $n$ and in every other $I_v$ the sum is non-zero.
It is not difficult to check that there is no $k$-star in this configuration.
This implies that for $d,n\ge 2$ and $k\le d$

\begin{equation}\label{claim:k-hub-lower-construction}
	\ex_d(\cSk,n)\ge \left\lfloor \frac{d}{d-k+1}\right\rfloor n^{d-1}+O_d(n^{d-2}).
\end{equation}

It is tempting to {\bf conjecture} that~\eqref{d-recursion} might give the correct value of  $\ex_d(\cS_d,n)$, at least if $n$ is sufficiently large. Note that~\eqref{d-estimate} yields $3(n-1)^2+3\le \ex_3(\cSt,n)\le  3(n-1)^{2}+9$.  We prove that in $3$-dimensions the lower bound is indeed the correct value.

\begin{theorem}
	\label{thm:3-dim-no-3-hubs}
Let $M$ be a $3$ dimensional $0$-$1$ matrix of size $n_1 \times n_2 \times n_3$   with $n_1,n_2,n_3\ge 2$.
If $M$ does not contain a $3$-star then it has at most
\begin{equation}\label{f3d}
(n_1-1)(n_2-1)+(n_1-1)(n_3-1)+(n_2-1)(n_3-1)+3
 \end{equation}
$1$-entries and this bound is optimal.
\end{theorem}	
The $\cS_3$-free construction witnessing optimality comes from~\eqref{d-recursion}, let us define it explicitly.
Let the $1$-entries be those elements of $B(n_1,n_2,n_3)$ that have exactly one zero coordinate plus additional three $1$-entries $(1,0,0)$, $(0,1,0)$, and $(0,0,1)$.
Setting $n_1=n_2=n_3$ we get
\begin{equation}\label{thm:33}
	\ex_3(\cSt,n)=3(n-1)^{2}+3 \text{ if } n\ge 2.
\end{equation}

\subsection{Visibility, axis parallel rays}
Our main Theorem~\ref{thm:k-hub} can be rephrased in the following way. If for every $1$-entry of $M$ there are at least $\ell=d-k+1$ rows containing only this $1$-entry then the maximum number of $1$-entries is $(d/\ell) n^{d-1}+o(n^{d-1})$.
Instead of single-entry lines we may consider axis parallel rays.
Given a point $p\in R^d$ the {\em ray} $R(p,v)$ is the set of all the points of the form $p+\lambda v$ where $v$ or $-v$ is one of the $d$ unit vectors of the standard basis and $\lambda \geq 0$.
We say that $M$ has the $L$-{\em visibility} property, if for each $1$-entry there are at least $L$ \emph{own rays}, in which directions the $1$-entry does not see any other $1$-entries. (Note that $L\leq 2d$, there are two directions parallel to each axis).

Let $g_d(L,n)$ denote the maximum weight of a $n\times n\times \dots \times n$ matrix $M$ of dimension $d$ having the $L$-visibility property.
 As $B(n)$ has $dn^{d-1}$ rows and in an $M$ with the $L$-visibility property each $1$-entry has $L$ own rays, $M$ has at most $2dn^{d-1}/L$ $1$-entries
\begin{equation}\label{claim:k-ray-upper}
  g_d(L,n)\le \frac{2d}{L}n^{d-1}.
 \end{equation}
The slight improvement, analogous to~\eqref{claim:k-hub-upper-d}, does not seem to be that easy.
Nevertheless it is true that~\eqref{claim:k-ray-upper} is asymptotically tight, it gives the correct order of the magnitude

\begin{theorem}\label{thm:weak-k-hub}
	$g_d(L,n)=  \frac{2d}{L} n^{d-1}+o(n^{d-1})$ if $d,n\ge 2$ and $1\leq L\le 2d$.
\end{theorem}
The proof is again probabilistic as for Theorem~\ref{thm:k-hub}.
Some cases are easy, e.g.,
\begin{equation}\label{eq:10}
g_d(1,n)=  n^{d}-(n-2)^{d}.
 \end{equation}

Indeed, for the lower bound one can take the boundary entries as $1$-entries to form the matrix $M$, while for the upper bound one can replace $1$-entries that are not on the boundary with boundary $1$-entries just like we did when proving~\eqref{claim:onestar}.

Also, for even $L$ we can use Theorem~\ref{thm:k-hub} and the obvious fact that
$$g_d(L,n)\geq \ex_d(\cS_{d+1-L/2},n)=(1+o(1)) \frac{d}{L/2} n^{d-1}. $$

We can rephrase this theorem in the language of forbidden submatrices. Let a {\em weak $K$-star} be a matrix with $(K+1)$ $1$-entries, one of them called a weak $K$-hub such that the remaining $K$ $1$-entries are neighboring this weak $K$-hub in some row. The difference compared to a $k$-star is that here we allow that two such $1$-entries are in the same row but on different sides of the hub $1$-entry. Let $\cS'_K$ be the family of all weak $K$-stars. Forbidding any member of $\cS'_K$ is equivalent to the previous problem with $L=2d-K+1$, so by Theorem \ref{thm:weak-k-hub}:
$$\ex_d(\cS'_K,n)= g_d(2d-K+1,n) =\frac{2d}{2d-K+1} n^{d-1}+o(n^{d-1})\text{ if }d,n\ge 2\text{ and }K\le 2d.$$

\section{Hypergraph packings}\label{sec:hypergraphs}

\subsection{Matchings}
A {\em hypergraph} $\cH$ is a family of sets $E(\cH)$ (called edges, or hyperedges) together with a (usually finite) set $V(\cH)$ (called vertices), such that each edge $e$ is a subset of $V(\cH)$. We denote $e(\cH)=|E(\cH)|$.
The hypergraph is $r$-{\em uniform} if every hyperedge has $r$ vertices.
The complete $r$-uniform hypergraph on the $n$ vertex set $A$ is denoted by $K_r(A)$, or sometimes $K_r(n)$.
It has exactly $\binom{n}{r}$ edges.
The  {\em degree} of a vertex $\deg_\cH(v)$ (resp., $\deg_\cH (B)$) is the number of hyperedges of $\cH$ that contain the vertex $v$ (resp., the subset $B\subset V(\cH)$). The co-degree of two vertices is also denoted by  $\deg_\cH(v,w):=\deg_\cH(\{v,w\})$.

A  {\em matching} $\cM$ in $\cH$ is a set of pairwise disjoint hyperedges, and let $\nu(\cH)$ denote its maximum cardinality, $\max |\cM|$.
So for $r$-uniform hypergraphs we have $\nu(\cH)\leq |V(\cH)|/r$.
One can find a large matching in a hypergraph $\cH$ (one which covers almost all vertices of $\cH$) if it is {\em almost regular} and {\em uncrowded}, i.e., it has small co-degrees. We will use the following result of Pippenger and Spencer (see also~\cite{franklrodl,furedimatchingsurvey}).

\begin{theorem}{\rm \cite{ps}}\label{thm:matching}
	For every $r\ge 2$ and $\eta>0$ there exist a $\eta'=\eta'(r, \eta)> 0$ and an $n_0$ such
	that the following holds for every $r$-uniform hypergraph $\cH$ on at least $n_0$ vertices. If 	 	
	\begin{itemize}
		\item $|\deg_\cH (v)- \Delta| \leq \eta'\Delta$ for every vertex $v\in V(\cH)$ and
		\item $\deg_\cH (v,w ) \le  \eta'\Delta$, for every pair $\{v, w\}\subset V(\cH)$
	\end{itemize}
	for some $\Delta$, then there is a matching in $\cH$ that leaves at most $\eta|V(\cH)|$ vertices uncovered, $\nu(\cH)\geq (1-\eta)|V(\cH)|/r$.
\end{theorem}

\subsection{Induced packings} 
 Given  $\cH$ and a subset $W\subset V(\cH)$.  The vertex set of the {\em induced subhypergraph} $\cH|W$ is $W$, and its edge set consists of all hyperedges of $\cH$ contained in $W$, $E(\cH|W):=\{e\in E(\cH): e\subset W \}$. Given an $r$-uniform hypergraph $\cF$ with $f$ edges, an {\em induced packing} of $\cF$ is a family of subsets $W_1, \dots , W_N$ of $V(\cH)$ such that each $\cH|W_i$ is isomorphic to  $\cF$ and these copies of $\cF$ are edge-disjoint.
Obviously, $N\leq |E(\cH)|/f$. E.g., if $n$ is even, $r=2$ and $\cH$ is the complete graph minus a perfect matching, $\{ e_1, \dots, e_{n/2}\}$ then the sets   $W_{i,j}:=e_i\cup e_j$ ($1 \leq i<j\leq n/2$) form an induced packing of $C_4$'s of size $\frac{1}{4} \binom{n}{2} +O(n)$.
Frankl and F\"uredi~\cite{FrFu-87} proved that similar large packings exist for every $r$-uniform hypergraph $\cF$.

\begin{theorem}\label{thm:induced1} {\rm \cite{FrFu-87}} {\rm [Maximum induced packings]}
Given an $r$-uniform hypergraph $\cF$ of  $f$ edges. For every $\delta>0$ there exists  an $n_0=n_0(\cF)$ such
	that the following holds.
For $n>n_0$ there exists a hypergraph $\cH$ on $n$ vertices, and subsets $W_1, \dots , W_N \subset V(\cH)$ such that
\\ ${}$\quad {\rm (i)}\enskip\enskip  each induced subhypergraph $\cH|W_i$ is isomorphic to $\cF$,
\\ ${}$\quad {\rm (ii)}\enskip  these copies are edge-disjoint, (we are having an induced packing),
\\ ${}$\quad {\rm (iii)} $|W_i\cap W_j| \le r$, for every pair $i \neq j$, and
\\ ${}$\quad {\rm (iv)}  $N> (1-\delta)\binom{n}{r}/f$.
\end{theorem}
Note that if $|W_i\cap W_j| = r$, then the set $(W_i\cap W_j)$ cannot be in $E(\cH)$.

An $r$-uniform hypergraph $\cH$ is {\em $k$-partite} with parts $V_1, \dots, V_k$ if the union of this partition is $V(\cH)$ and
 its hyperedges contain at most one vertex from each part. (We call these kind of subsets partial transversal sets).
If $\cH$ contains all of these $r$-sets then it is denoted by  $K_r(V_1, \dots , V_k)$ and called a complete $k$-partite $r$-uniform hypergraph.
In case of $n =|V_1|= \dots = |V_k|$ we use the shorthand $K_r^k(n)$, its size is $\binom{k}{r}n^r$.
The main result of this Section is to prove a $k$-partite version of Theorem~\ref{thm:induced1}.

\begin{theorem}\label{thm:induced_new}  {\rm [Maximum induced $k$-partite packings]}
Given an $r$-uniform $w$-partite hypergraph $\cF$ of $f$ edges with parts $U_1, \dots, U_w$, $b=|\cup U_j|$, $k\geq w$.
For every $\delta>0$ there exist an $n_0=n_0(\cF,k)$ such that the following holds.
For $n>n_0$ there exists an $r$-uniform $k$-partite $\cH$ on $k\times n$ vertices, with parts $V_1, \dots, V_k$, each of size $n$, and
 $b$-subsets $W_1, \dots , W_N \subset V(\cH)$ such that
\\ ${}$\quad {\rm (i)}\enskip\enskip each induced subhypergraph $\cH|W_i$ is isomorphic to $\cF$ with a {\em partition preserving} embedding, that is,
 there is an injection $\alpha_i:\{1,2,\dots, w\} \to \{1,2,\dots, k\}$ and another injection $\varphi_i: \cup U_j\to V(\cH)$ such that set $\varphi_i(U_j)$ is a subset of $V_{\alpha_i(j)}$ for each $j$, and  $\varphi_i(e)\in E(\cH)$ for each $e\in\cF$,
\\ ${}$\quad {\rm (ii)}\enskip  these copies are edge-disjoint, (we are having an induced packing),
\\ ${}$\quad {\rm (iii)} $|W_i\cap W_j| \le r$, for every pair $i \neq j$, and
\\ ${}$\quad {\rm (iv)}  $N> (1-\delta)\binom{k}{r}n^r/f$.
\end{theorem}

\subsection{The existence of generalized transversal designs}

For the proof of Theorem~{\rm \ref{thm:induced_new}}
we adapt (and simplify) the argument from~\cite{FrFu-87}.

A crucial notion we need is the following:
\begin{definition}
	A  {\em generalized transversal design} $T(n,k,u,r,d,\gamma)$, where $n,k,u,r,d$ are positive integers and $\gamma\geq 0$ is as follows. It is a $ku$-uniform family $\cQ$ on $kn$ elements, with parts  $V_1, \dots, V_k$, each of size $n$, with $|Q\cap V_i|=u$ for all $Q\in \cQ$ and $1\leq i \leq k$, $|Q\cap Q'|\leq r$ for distinct $Q,Q'\in \cQ$, and	
	$$|\deg_\cQ(R)-d|\leq \gamma$$ for all $r$-subsets  $R$ meeting $r$ distinct parts.	The value $\gamma$ is called the {\em tolerance} of $T$.
\end{definition}

\begin{lemma}\label{lemma:0}{\rm [The existence of designs with small tolerance]}
Suppose that $k,u,r$ are fixed, $d\leq n^{1/3}$,  $5\sqrt{kr d\log n}\leq \gamma \leq 2d$, and $n$ is sufficiently large
 $(n> n_0(k,u)$, $\gamma > 12r)$. Then there exists a generalized transversal design $T(n,k,u,r,d,\gamma)$.
\end{lemma}

The existence of perfect transversal designs is a well-known research topic of combinatorics.
Very likely a more precise theorem can be proved by the advanced methods of~\cite{GKLO}
(by iterative absorption) and~\cite{Keevash-I,Keevash-II} (by randomised algebraic constructions)
but for our purpose the present weaker version is sufficient,

To prove Lemma \ref{lemma:0} we use only standard random methods. We need Chernoff's bound in the following form (see e.g.,~\cite{BBbook} page 12). 
An indicator (or Bernoulli) variable $Y_i$ is a random variable taking only the values $0$ and $1$ with probabilities
 $\text{\rm Prob} (Y_i=1) =p_i$ and $\text{\rm Prob} (Y_i=0) =1-p_i$.

\begin{theorem}\label{thm:chernoff}
	Suppose that $Y$ is the sum of mutually independent indicator random variables $Y_1, \dots, Y_N$
with probabilities  $0\leq p_1, \dots, p_N\leq 1$,   $\mu:=\EE(Y)$, and $0\leq \varepsilon \leq 1$.
Then $$\text{\rm Prob}\big[|Y-\mu|>\varepsilon \mu\big]<2e^{-\varepsilon^2 \mu/3}.$$
\end{theorem}

\begin{proof}[Proof of Lemma~{\rm \ref{lemma:0}}]
	
Recall that  $K_r^k(n)$ is an $r$-uniform hypergraph with vertex set $V:=V_1\cup \dots \cup V_k$
 where $|V_1|= \dots = |V_k|=n$ and its edges are the $r$-element transversals, so $e(K_r^k(n))=\binom{k}{r}n^{r}$.
Define $K^{k,u}(n)$ on the same vertex set with $ku$ element edges $e$ such that $|e\cap V_i|=u$ for all $1\leq i\leq k$,
 so $e(K^{k,u}(n))=\binom{n}{u}^{k}$.
Define a probability space on $E(K^{k,u}(n))$, choose each edge randomly and mutually independently with probability
  $p:=d/\binom{n-1}{u-1}^{r}\binom{n}{u}^{k-r}$.
Let $\cQ$ denote the obtained random subfamily.
For every edge  $R\in E(K_r^k(n))$ we get $ E(\deg_\cQ(R))=d$.

\begin{claim}\label{Claim1:r-sets}
The probability that there exists any $R\in E(K_r^k(n))$ with $|\deg_\cQ(R)-d|>\frac{1}{2}\gamma$ is less than $1/3.$
\end{claim}
\begin{proof} 
For each $R$ Theorem~\ref{thm:chernoff} implies (with $N:=\deg(R)$ in $E(K^{k,u}(n))$, i.e., $N=\binom{n-1}{u-1}^{r}\binom{n}{u}^{k-r}$, $\mu:=d$, and $\varepsilon:=\gamma/(2d)$, therefore $0\le \varepsilon \le 1$) that 
$$
\text{\rm Prob}\big[|\deg_\cQ(R)-d|>\frac{1}{2}\gamma\big]<2e^{-\gamma^2 /(12d)}.$$
	Since $\gamma^2 \ge 25 krd\log n> 12d(r+1) \log n $, the right hand side is smaller than $2 n^{-(r+1)} \leq \frac{1}{3} 2^{-k} n^{-r}\leq 1/(3 e(K_r^k(n)))$
	for $n\geq 6\times  2^{k}$.	Taking the union bound for every $R$ finishes the proof. 
\end{proof}
In fact, the right hand side here is $o(1)$ as $n\to \infty$, but we do not need that.
Similar bounds can be shown to all $r$-subsets of $V$, but we do not need those either.

\begin{claim}\label{Claim2:6r+2-sets}
The probability that there exists a $\cQ'\subset \cQ$ of size $|\cQ'|=6r+2$ with a relatively small union
 $|\cup_{Q\in \cQ'} Q|\leq (ku-r)(6r+2)-(2r+1)$  is less than $1/3.$
\end{claim}
\begin{proof}
Call a  family $\cC$ an $(a,b,c)$-system if $|\cup \cC|\leq a$, $|\cC|=b$, and $|C|=c$ for each $C\in \cC$.
We give a very generous upper bound for the number of $(a,b,c)$ systems on $v$ vertices.
It is at most
$$  \binom{v}{a}\binom{\binom{a}{c}}{b} < v^a \binom{a}{c}^b < v^a 2^{ab}.
$$
We use this formula with $v:=kn$, $a:=(ku-r)(6r+2)-(2r+1)$,  $b:=(6r+2)$, and $c:=ku$.
The probability that a given configuration is occurring in $\cQ$ is $p^b$ (or $0$),
 so the total probability in Claim~\ref{Claim2:6r+2-sets} is at most  $p^b v^a 2^{ab}$.
Since $p^b=O(d^{6r+2}/n^{(ku-r)(6r+2)})$ and $v^a=O(n^{(ku-r)(6r+2)-(2r+1)})$ this is at most
 $\alpha(k,u)\times (d^{6r+2}/n^{2r+1})$, where $\alpha(k,u)$ is a constant depending only on these two variables.
This is $o(1)$ when $n\to \infty$ since $d\leq n^{1/3}$.
\end{proof}

Choose a family $\cQ\subset E(K^{k,u}(n))$ at random as described in Claim~\ref{Claim1:r-sets}.
The sum of probabilities in Claim~\ref{Claim1:r-sets} and in Claim~\ref{Claim2:6r+2-sets} is less than $1$ so we may suppose that
 there exists a family $\cQ$ avoiding all small dense parts  described in  Claim~\ref{Claim2:6r+2-sets} and satisfying
 $$|\deg_\cQ(R)-d| \leq \frac{1}{2}\gamma$$
for all  $R\in E(K_r^k(n))$.

Call a pair of sets $Q,Q'\in \cQ$ a {\em bad pair} if $|Q\cap Q'|\geq r+1$.
Remove the bad pairs from $\cQ$ consecutively till no more are left.
Suppose that $Q_1,Q_1'$, $Q_2,Q_2', \dots$, $Q_m,Q_m'$ are bad pairs, they form the hypergraph $\cB\subset \cQ$,  $|\cB|=2m$.
Define the (multi)hypergaph $\cM$ with $m$ edges of the form  $Q_i\cup Q_i'$. Note that each such union has at most $2ku-r-1$ elements.
We claim that for each $r$-set $R$ one has $\deg_\cM(R)\leq 3r$.

Suppose, on the contrary, that $R$ is contained in $Q_1\cup Q_1'$, $Q_2\cup Q_2', \dots$, $Q_{3r+1}\cup Q_{3r+1}'$.
Consider their union
\begin{multline*}
  |\cup_{1\leq i\leq 3r+1} (Q_i\cup Q_i')| \leq |R| + \sum_i  |Q_i\cup Q_i'\setminus R| \\
   \leq r+ (3r+1)(2ku-2r-1)= (ku-r)(6r+2)-(2r+1).
\end{multline*}
This contradicts the choice of $\cQ$, (namely that it avoids the configurations in Claim~\ref{Claim2:6r+2-sets}).

The upper bound  $\deg_\cM(R)\leq 3r$ implies that $\deg_\cB(R)\leq 6r$.  This is less than $\gamma /2$, so the family
 $\cQ' :=\cQ\setminus \cB$ is indeed a $T(n,k,u,r,d,\gamma)$ transversal design.
\end{proof}

\subsection{Proof of Theorem~{\rm \ref{thm:induced_new}}}

It is enough to deal with the case when $\cF$ is a {\em balanced} $r$-uniform $k$-partite hypergraph.
A hypergraph $\cG$ with parts $W_1, \dots, W_k$ is called balanced if $|W_1|= \dots =|W_k|$.
Given a $w$-partite hypergraph $\cF$ of $f$ edges with parts $U_1, \dots, U_w$ ($w\leq k$), one can turn it to a balanced hypergraph $\cF'$
 with a vertex set $W'$ by adding extra isolated vertices.
If we can find $(1-\delta)\binom{k}{r}n^r/f$ copies of $\cF'$ induced by $W'_1, \dots$  as described above in Theorem~\ref{thm:induced_new},
 then obviously,  each $W'_i$ can be pruned into an induced copy of $\cF$.

From now on, we suppose that $\cF$ is balanced, so $u:=|U_1|= \dots =|U_k|$.
Consider  $K_r^k(u)$ with parts $U_1', \dots, U_k'$ and fix $u_i'\in U_i'$ for $1\leq i \leq r$.
Let $\beta$ denote the number of subhypergraphs $\cF'\subset E(K_r^k(u))$ such that $\{u_1', \dots, u_r'\}\in \cF'$ and $\cF'$ is isomorphic to $\cF$ with a partition preserving isomorphism (i.e., one satisfying Condition~(i) of Theorem~{\rm \ref{thm:induced_new}}).
Obviously, this $\beta$ is the same for every other edge of $K_r^k(u)$.
So every $K_r^k(u)$ contains exactly $\beta \binom{k}{r} u^r/f$ partition preserving copies of $\cF$. In addition, since $\cF$  can be embedded into  $K_r^k (u) $, we have  $1\leq \beta  \leq  \binom{  \binom{k}{r} u^r  -1 }{f-1}$, note that the upper bound depends only on $\cF$.

Define $d:=\lfloor n^{1/3}\rfloor$, $\gamma:=\lceil n^{1/5}\rceil$. By Lemma~\ref{lemma:0}
 there exists a generalized transversal design $\cQ$ with parameters $(n,k,u,r,d,\gamma)$ for large enough $n$.
We know that
\begin{equation}\label{eq12} 
    {d-\gamma\leq  \deg_\cQ(R)\leq d+\gamma}
  \end{equation} 
for all $R\in E(K_r^k(n))$.

Define a probability space on $E(K_r^k(n))$, choose each edge randomly and mutually independently with probability
  $q$ (its value is defined later as a small constant independent of $n$).
Let $\cR_0$ denote the obtained random subfamily and let $\cR:=E(K_r^k(n))\setminus \cR_0$.
We have $E(|\cR_0|)= q\times e(K_r^k(n)) = q \binom{k}{r} n^r$.
Then
Theorem~\ref{thm:chernoff} implies (with $\varepsilon:=1$) that
\begin{equation}\label{eq13} \text{\rm Prob}\big[|  \cR_0|> 2q \binom{k}{r} n^r \big]
 < 2\exp \big[- q \binom{k}{r} n^r/3 \big]=o(1).
 \end{equation}

Given $\cR$ and $\cQ$ define an $f$-uniform hypergraph $\cJ$ as follows.
The vertex set $V(\cJ):=E(\cR)$ and the hyperedges are those copies of $\cF$ which are spanned by a $Q\in \cQ$
 by a partition preserving isomorphism, in particular $\cF \sim \cR|Q$.
The probability that a given $R\in \cR$, $R\subset Q\in \cQ$ belongs to a hyperedge $J\in \cJ$ generated by $Q$ is exactly
$$
  h:=  \beta \times  (1-q)^{f-1} q^{\binom{k}{r} u^r -f}.
  $$
For each $R\in \cR$ we have $E(\deg_\cJ(R)) = h\deg_\cQ(R)$. Let
   $\gamma_1(R)=\sqrt{3h(r+1)\log n \deg_\cQ(R) }$. For large $n$ $\gamma_1(R)$ is less than $h\deg_\cQ(R)$, in fact it is $o(h\deg_\cQ(R))$ as $n\to \infty$ using \eqref{eq12} and that $d=\lfloor n^{1/3}\rfloor$, $\gamma=\lceil n^{1/5}\rceil$.

Theorem~\ref{thm:chernoff} implies that
\begin{multline} \text{\rm Prob}\big[\exists R \in E(K_r^k(n)), R\in \cR \enskip \text{\rm such that }|  \deg_\cJ(R)- h \deg_\cQ(R) |>\gamma_1(R)\big] \label{eq14}\\
 <\binom{k}{r}n^r \times 2\exp \big[-(\gamma_1)^2 /(3h\deg_\cQ(R))\big] = O(1/n)=o(1).
 \end{multline}
 The sum of probabilities in~\eqref{eq13} and~\eqref{eq14} is $o(1)$, so using ~\eqref{eq12} we get that for large enough $n$ there is a choice of $\cR$ such that for every $R$
\begin{equation}\label{eq15}  |\deg_\cJ(R)- h d |\leq h\gamma + \gamma_1.
 \end{equation}

Notice that a matching $\cM$ in $\cJ$, generated by the hyperedges $W_1, \dots, W_N\in \cQ$, yields edge-disjoint, induced copies of $\cF$ in $\cR$.
Let us check if the conditions of Theorem~\ref{thm:matching} hold.
Recall that $\delta$ was a fixed parameter in the theorem. Choose $q$ such that $q<\delta/4$, set $\Delta:=hd$ and define $\varepsilon:=\eta'(f, \delta/2)$.
Equation~\eqref{eq13} implies that $|V(\cJ)|=|E(\cR)|\geq (1-\delta/2) \binom{k}{r} n^r$. We want to apply Theorem~\ref{thm:matching} on $\cJ$ with $``\mu"$ being $\delta/2$ and $``r"$ being $f$.
First, ~\eqref{eq15} implies that in $\cJ$ the degree of every $R$ differs from $\Delta$ by at most $h\gamma+\gamma_1\le \varepsilon hd =\varepsilon \Delta$ for large $n$ (as $\varepsilon$ is a constant, $d=\lfloor n^{1/3}\rfloor$, $\gamma=\lceil n^{1/5}\rceil$, and $\gamma_1$ is small due to $\gamma_1=o(h\deg_\cQ(R))$ combined with~\eqref{eq12}), as required.
Second, since $|Q\cap Q'|\leq r$ for distinct $Q,Q'\in \cQ$ in the transversal design $\cQ$, we get that for every
co-degree we have $\deg_{\cJ}(R, R')\leq 1\leq \varepsilon \Delta=\varepsilon hd$ for large $n$ (as $\varepsilon$ and $h$ are constants and $d=\lfloor n^{1/3}\rfloor$), as required.
Thus we can indeed apply Theorem~\ref{thm:matching} to get a matching of size at least
$$ N\geq (1-\delta/2) |V(\cJ)|/f\geq (1-\delta/2)^2 e(K_r^k(n))/f\geq (1-\delta)\binom{k}{r}n^r/f
  ,$$ which gives this many edge-disjoint induced copies of $\cF$ in $\cR$, finishing the proof of Theorem \ref{thm:induced_new}.

\section{The general case, Proofs of Theorem~\ref{thm:k-hub} and~\ref{thm:weak-k-hub}}\label{sec:general}

\begin{proof}[Proof of Theorem~{\rm \ref{thm:k-hub}}]
	The upper bound follows from~\eqref{claim:k-hub-upper}. We have a matching lower bound by~\eqref{claim:k-hub-lower-construction} whenever $d$ is divisible by $d-k+1$ including the cases $k=1$  (see~\eqref{claim:onehub})
 and $k=d$ (see~\eqref{d-recursion}). We are left to show that one can find a matrix of size $n$ and dimension $d$ with sufficiently many $1$-entries for the case $1<k<d$.
Recall that the problem is equivalent to finding a large enough subset $P\subset B_d(n)$
 such that every $x\in P$ has $\ell:=(d-k+1)$ axis parallel lines containing $x$ and no other element of $P$ (i.e., there are (at least) $\ell$ single-entry lines through $x$), $2\leq \ell\leq d-1$.
We construct an appropriate $P$ as follows.

Let $F:=\cF_{d,\ell}$ be the unique $(d-1)$ uniform hypergraph having $d$ vertices and $\ell$ hyperedges.
Recall that  $K_{d-1}^d(n)$ is a $(d-1)$-uniform hypergraph with vertex set $V_1\cup \dots \cup V_d$
 where $|V_1|= \dots = |V_k|=n$ and its edges are the $(d-1)$-element transversals, so $e(K_{d-1}^d(n))=dn^{d-1}$.
According to Theorem~\ref{thm:induced_new} for every $\delta>0$ there exists an $n_0=n_0(d)$ such that for $n>n_0$ there exists a subhypergraph
$\cH\subset K_{d-1}^d(n)$ of size at least  $(1-\delta) dn^{d-1}$ and
 subsets $W_1, \dots , W_N \subset V(\cH)$ such that
 each induced subhypergraph $\cH|W_i$ is isomorphic to $F$,
 these copies are edge-disjoint, and $E(\cH)=\cup E(\cH|W_i)$.

Given such a collection of copies of $F$, there is a corresponding set of lattice points $P\subset B_d(n)$ such that
 $p_i=(x_1,\ldots ,x_d)\in P$ if and only if the vertex set $W_i$ meets $V_\alpha$ in the $x_i$th vertex ($1\leq \alpha\leq d$).
 The size of $P$ equals to the number of copies of $F$, i.e., it is at least  $(1-\delta) dn^{d-1}/\ell$.

 We claim that for every $p_i\in P$ there are $\ell$  $1$-entry  lines containing it. Indeed,
 suppose that $E(\cH|W_i)=\{ W_i\setminus x_\alpha: \alpha\in A_i \}$, where $|A_i|= \ell$ a set given by the embedding of  $F$ into $W_i$.
 Consider a hyperedge $h:=W_i\setminus x_\alpha$ corresponding to $\cH|W_i$, (here $\alpha\in A_i$).
 Suppose, on the contrary, that the line $r(p_i, \alpha)$ contains another lattice point $p_j\in P$. Then $h=W_i\cap W_j$.
 Since $h\in E(\cH|W_i)\subset E(\cH)$, the edge sets $E(\cH|W_i)$ and $E(\cH|W_j)$ are not disjoint, a contradiction. 
\end{proof}

\begin{proof}[Proof of Theorem~{\rm \ref{thm:weak-k-hub}}]
	The upper bound follows from~\eqref{claim:k-ray-upper}. We have a matching lower bound
for even $L$ using Theorem~\ref{thm:k-hub}, and also for $L=1$ by~\eqref{eq:10}. We are left to show that one can find a
large enough subset $P\subset B_d(n)$ for the case $L=2\ell+1, \, 1\leq \ell<d$.

Let $F'=\cF'_{d,\ell}$ be a $(d-1)$-uniform $d$-partite hypergraph
 with parts $U_1, \dots, U_d$, $|U_j|=1$ for $1\leq j \le d-1$, $U_d=\{ v_1,v_2\}$, $U:=\cup U_i$, and having $2\ell+1$  hyperedges, namely
the $(d-1)$-element sets of the form $U\setminus U_\alpha \cup \{v_j\}$ for   $1\leq \alpha\leq \ell$ and $j\in \{1,2\}$  and the set  $U\setminus U_d$.
According to Theorem~\ref{thm:induced_new} for every $\delta>0$ there exists an $n_0=n_0(d)$ such that for $n>n_0$ there exists a subhypergraph $\cH'\subset K_{d-1}^d(n)$ of size at least  $(1-\delta) dn^{d-1}$ and
 subsets $W_1, \dots , W_N \subset V(\cH)$ such that
 each induced subhypergraph $\cH|W_i$ is isomorphic to $F'$,
 these copies are edge-disjoint, and $E(\cH)=\cup E(\cH|W_i)$. Suppose $\varphi_i: U\to V(\cH)$ yields $\cH|W_i$.

 Given such a collection of copies of $F'$,
 one can define a corresponding set of lattice points $P\subset B_d(n)$ such that
each $W_i$ gives exactly two lattice points, since each $W_i$ contains two $d$-transversals.
Since $|W_i\cap W_j|<d$, the size of $P$ is twice the number of copies of $F'$, i.e., it is at least  $(1-\delta) 2dn^{d-1}/(2\ell+1)$.

 We claim that for every $p\in P$ there are $\ell$  $1$-entry  lines containing it. Indeed,
if  $p=\varphi_i(U\setminus U_d \cup \{v_j\})$ then
  for each $1\leq \alpha\leq \ell$ the row determined by $\varphi_i(U\setminus U_\alpha\setminus U_d \cup \{v_j\})$ is a $1$-entry  line.
  This gives $2\ell$ rays emanating from $p$.
Furthermore, the line defined by $\varphi_i(U\setminus U_d)$ is a line with exactly two $1$-entries, it covers only
  $\varphi_i(U\setminus U_d \cup \{v_1\})$ and $\varphi_i(U\setminus U_d \cup \{v_2\})$, so we got one more free ray from $p$.
\end{proof}

\section{$3$-dimensional matrices, The proof of Theorem~\ref{thm:3-dim-no-3-hubs}}\label{sec:3dim}

We introduce some definitions specific to $3$ dimensional matrices. Let $P$ be the set of $1$-entries of $M$.
We refer to the row containing an entry $p\in M$ that has the first, second, third coordinate non-fixed as $X(p)$, $Y(p)$, $Z(p)$, respectively. We refer to the $2$ dimensional cross section through $p$ that has the first, second, third coordinate fixed as $YZ(p)$, $XZ(p)$, $XY(p)$ and call such cross sections as a $YZ$-cross section, $XZ$-cross section, $XY$-cross section, respectively.
Given an entry $p=(a,b,c)\in M$ we also use the notation $(*,b,c)$ for $X(p)$, etc, whichever notation is more convenient.

Let $f(n_1, n_2, n_3)$ denote the function defined in~\eqref{f3d}. We have to show that the weight of an $\cS_3$-free matrix $M$ is at most $f(n_1, n_2, n_3)$ (whenever every $n_i\geq 2$). 

We prove by induction on $\sum n_i$. If $n_1=n_2=n_3=2$ then the theorem holds. If $n_1\ge 3$ and there is a $YZ$-cross section containing at most $(n_2-1)+(n_3-1)$ $1$-entries then we can delete this cross section from $M$ and apply induction because
$$ f(n_1-1, n_2, n_3)+(n_2-1)+(n_3-1)= f(n_1, n_2, n_3).
$$
Let us assume that the theorem holds for every matrix $M'$ which has size smaller than $n_1 \times n_2 \times n_3$ (but all dimensions of $M'$ are at least $2$).
Thus we can assume that there is no $YZ$-cross section with at most $(n_2-1)+(n_3-1)$ $1$-entries if $n_1\ge 3$. Similarly, we can assume that each $XY$-cross section has at least $(n_1+n_2-1)$ $1$-entries if $n_3\ge 3$ and that each $XZ$-cross section has at least $(n_1+n_3-1)$ $1$-entries if $n_2\ge 3$.

From the inductional assumption it follows that in $M$ there cannot be two parallel rows in the same cross section such that both have only $1$-entries. Suppose, on the contrary, e.g.,  that $Y(p)$ and $Y(q)$ are two full $1$ rows, $Y(p), Y(q)\subset YZ(p)$. Then $\cS_3$-freeness implies that $XY(p)$ has no $1$-entry outside $Y(p)$. So this cross section contains only $n_2$ non-zero entries, a contradiction.

\smallskip
For an entry $p$ in the matrix, let $t(p)$ be the number of rows containing $p$ that have at most one $1$-entry. Let $t$ be the maximum of $t(p)$ over every $0$-entry $p$. With a possible reordering of the layers of the matrix we can assume that $(0,0,0)$ is a $0$-entry and that $t(p)=t$, that is, no $0$-entry $p$ has $t(p)>t(0,0,0)$.  Note that $1\le t \le 3$.

Fix an injection $\lambda$ from $P$ into the $1$-entry rows of $M$ such that the $1$-entry row $\lambda(p)$ contains $p$, i.e.,  it is $r(p,i)$ for some $i\in \{1,2,3\}$. Such $\lambda$ exists because $M$ is $\cS_3$-free.

\smallskip

\noindent \textbf{Case $t=1$.}

In this case the single-entry lines associated with each $1$-entry are pairwise disjoint.
These $|P|$ lines cover at least  $|P|n_3$ entries of $M$ (whenever $n_3=\min\{ n_1,n_2,n_3\}$). Hence
\begin{equation*}
  |P|\le n_1\times n_2 =(n_1-1)(n_2-1) +(n_1-1)+(n_2-1) +1 < f(n_1, n_2, n_3).
 \end{equation*}

\smallskip
\noindent \textbf{Case $t=3$.}

In this case the three axes $X$, $Y$, and $Z$ are having a single $1$-entry each. Call these entries {\em special} and their 3-element set is denoted by $S$.
Let $\cL$ denote the set of all rows of types $(a,b,*)$ or $(a,*,c)$ or $(*,b,c)$  where $1\le a\le n_1-1, 1\le b\le n_2-1, 1\le c\le n_3-1$.
The size of  $\cL$  is $(n_1-1)(n_2-1)+(n_1-1)(n_3-1)+(n_2-1)(n_3-1)$.
To complete the proof of this case we define an injection $\varphi: P\setminus S \to \cL$.

	For each $1$-entry $p$ with no $0$-coordinate $\varphi(p):= \lambda(p)$.
For each $1$-entry $p$ with exactly one $0$-coordinate we assign to $p$ the row we get by replacing its $0$-coordinate with a $*$, i.e.,
 for $M(a,b,0)=1$ we have $\varphi(p):= (a,b,*)$,
 for $M(a,0,c)=1$ we have $\varphi(p):= (a,*,c)$, and
 for $M(0,b,c)=1$ we have $\varphi(p):= (*,b,c)$,
(here again $1\le a\le n_1-1, 1\le b\le n_2-1, 1\le c\le n_3-1$).
This assignment is an injection, as required, because $\lambda(p)\neq \lambda(q)$, and the set of rows assigned to the $1$-entries of
  $YZ(0,0,0)$, $XZ(0,0,0)$, $XY(0,0,0)$ are again distinct, and finally, if $(a,b,*)$ is assigned to $(a,b,0)$ then it is not a
single-entry line assigned to another $(a,b,c)$ with $a,b,c>0$.

\smallskip	
\noindent \textbf{Case $t=2$.}	
	
In this case with a possible reordering of the axes we can assume that $Y(0,0,0)$, and $Z(0,0,0)$ contain at most one $1$-entry, we call these special and put them into $S$. Let $Q$ be the set of $1$-entries on $X(0,0,0)$, we have $|Q|\geq 2$, so $n_1\geq 3$.
We define again an assignment $\varphi$. For each $1$-entry $p$ with no $0$-coordinate $\varphi(p):= \lambda(p)$. Notice that this assignment $\varphi: P\setminus (S\cup Q) \to \cL$ is so far an injection, as required.	

Our strategy is that for the remaining $1$-entries in $Q$ of the form $(a,0,0)$ we assign a row $\ell$, $\ell\in \cL$, $\ell\subset YZ(a,0,0)$, if there exists such a row. I.e., $\ell$ is of the form $(a,b,*)$ or $(a,*,c)$ with some positive $b$ or $c$ such that $\ell\notin \varphi(P\setminus (S\cup Q))$. By extending $\varphi$ with these assignments, it remains to be an injection.
We claim that all but at most one element of $Q$ can be assigned. The proof of this completes the proof of Theorem~\ref{thm:3-dim-no-3-hubs}.

Let $p=(a,0,0)\in Q$ arbitrary. Since $|Q|\ge 2$, there exists a $p'\in Q$, $p'\neq p$ and as there is no $\cS_3$ in $M$, either $Y(p)$ or $Z(p)$ is a single-entry line.
If $Y(p)$ is a single-entry line, and each $(a,b,*)\in \varphi(P\setminus (S\cup Q))$ then these are also single-entry lines. So the weight of $YZ(p)$
is exactly $(n_2-1)+|Z(p)|$. So if we cannot assign an appropriate line $\ell$ to $p$ then  $(n_2-1)+|Z(p)|\geq n_2+n_3-1$ by the induction hypothesis, so $Z(p)$ is a full $1$ row.
Since there are no two parallel full $1$ rows, this $p$ is unique in $Q$ (if it exists at all).
The case of  $Z(p)$ is a single-entry line is analogous.

We conclude that $\varphi$ can be extended to $Q$ apart from two entries (if they exist), namely a $p=(a,0,0)\in Q$ and  $p'=(a',0,0)\in Q$ where
 $Z(p)$ is a row with only $1$-entries while $Y(p)$ and every $(a,b,*)$ lines are single-entry lines,
 and  $Y(p')$ is a row with only $1$-entries and $Z(p')$ and every $(a',*,c)$ lines are single-entry lines ($1\le b\le n_2-1, 1\le c\le n_3-1$).

We claim that $p$ and $p'$ cannot exist simultaneously.
Let $r=(a,b,c)$ be a $1$-entry in $YZ(p)$ ($b,c>0$). As the $1$-entry $s=(a,0,c)$ is not a 3-hub, $X(s)$ must be all $0$-entries except for $s$. We obtained that  the $0$-entry $s'=(a',0,c)$ is contained in three  single-entry lines, namely  $X(s')$  $(=X(s))$,
 $Y(s')$  $(=(a',*,c))$, and  $Z(s')$  $(=Z(p'))$. This contradicts our condition $t\leq 2$.

If any of the $p$ and $p'$ exists we add it to the set of special vertices $S$, and got at most three special vertices and an injection $\varphi: P\setminus S \to \cL$.
\qed

\section{Conclusion}

\subsection{Related problems}\label{sec:further}

\paragraph{Forbidden subhypergraphs:}
As we have seen, ordering of the matrix is not relevant when forbidding every $k$-star. One can regard a $d$ dimensional $0$-$1$ matrix as the adjacency matrix of a $d$-partite $d$-uniform hypergraph. This way, e.g., $\ex_d(\cSd,n)$ is equivalent to the Tur\'an-type problem of determining the maximum number of hyperedges in a $d$-partite $d$-uniform hypergraph with every part having size $n$ that avoids a certain hypergraph as a subhypergraph. This forbidden hypergraph is the equivalent of a $d$-star, more precisely, it has a hyperedge $p$ (its hub) and $d$ further hyperedges, for each vertex of $p$ there is one that differs from $p$ in exactly that vertex.
F\"uredi and Özkahya considered $n$-vertex $\mathcal{F}$-free graphs and computed the Turán number $ex(n,\mathcal{F})$ for various collections of forbidden subhypergraphs $\mathcal{F}$ \cite{FuOz-11, FuOz-17}.


\paragraph{Star Avoidance Problem:}
A pattern avoidance problem in a $d$-dimensional $0$-$1$ matrix $M$ is called a star avoidance problem when the pattern is a star.
A set $S$ of vertices of a graph $G$ is a \textit{neighborhood (nbd) transversal} of $G$ if  $N[v] \cap S \neq \emptyset$, for each $v\in V(G)$, where $N[v]$ is the closed neighborhood of $v$. A \textit{nbd transversal problem} is to find a minimum cardinality nbd transversal of $G$ (also called the \textit{domination problem}). We next define a related notion.
A set $S$ of vertices of $G$ is a \textit{stable star transversal} of $G$ if for every $v\in V(G)$, $S$ intersects $I\cup\{v\}$ for every maximal stable subset $I$ of $N(v)$, the open neighborhood of $v$.
A \textit{stable star transversal problem} is to find a minimum cardinality stable star transversal of $G$. One can easily demonstrate that the stable star transversal problem is NP-complete by using the fact that the domination problem is NP-complete for bipartite graphs.

The entries of $M$ can be also regarded as the vertices of the Cartesian product of $d$ complete graphs, $K_{n_1} \square \dots \square K_{n_d}$. In particular, excluding all $1$-stars in $M$ is equivalent to having an independent set in this graph. Notice that any vertex $v$ of this graph together with any maximal stable subset of its open neighborhood is a star in $M$ and every star has this form. Therefore, the stable star transversal problem in $K_{n_1} \square \dots \square K_{n_d}$ is about finding a minimal size $S$ that intersects every star in $M$, thus the complement of $S$ avoids every star and is maximal such and so its size is equal to $\ex_d(\cSd,n_1,\dots,n_d)$. The stable star transversal problem was solved for $d=2$ by Manuel et al \cite{MaBr-23} and was left open problem for $d \geq 3$. Theorem \ref{thm:3-dim-no-3-hubs} solves the case $d=3$ and Claim \ref{claim:d-hub} gives an estimate for $d\ge 4$. The exact answer remains an open problem for $d \geq 4$.

\paragraph{Geodesic Avoidance Problem:}
A pair of 1-Entries $(x_1, \dots x_{i-1}, x_i, x_{i+1} \dots x_d)$ and $(x_1, \dots x_{i-1}, y_i, x_{i+1} \dots x_d)$, $x_i \neq y_i$, is called an \textit{$i$-dimensional edge} in a $d$-dimensional $0$-$1$ matrix $M$. A  path of edges in $M$ is said to be a \textit{dimension-distinct geodesic} if no two edges in the path are of the same dimension. A dimension-distinct geodesic of length $k$ in $M$ is said to be a \textit{dimension-distinct $k$-geodesic}.
A pattern avoidance problem in  $M$ is called a \textit{dimension-distinct $k$-geodesic avoidance problem} when the pattern is a dimension-distinct $k$-geodesic.

A geodesic (i.e., a shortest path) in a graph $G$ is maximal if it is not contained (as a subpath) in any other geodesic of $G$. A set $S$ of vertices of $G$ is a geodesic transversal of $G$ if every maximal geodesic of $G$ contains at least one vertex of $S$.
The geodesic transversal problem of $G$ is to find the minimum cardinality
of a geodesic transversal of $G$.
A geodesic packing of a graph $G$ is a set of vertex disjoint maximal geodesics in $G$. The geodesic packing problem of $G$ is to find the maximum cardinality of a geodesic packing of $G$.
The geodesic transversal problem and the geodesic packing problem are dual min-max invariant problems \cite{MaBr-23}. Note that for $k=2$ the maximal geodesics are the stars, Manuel et al \cite{MaBr-23} solved the geodesic packing problem for $K_{n_1} \square K_{n_2}$ by solving $\ex_2(\cS_2,n)$ and left the geodesic packing problem open for $d \geq 3$. Note that for $d\geq 3$ maximal geodesics are not stars anymore.

\paragraph{Binary covering codes of radius 1:}
The $k=d$ case of the star avoidance problem can be regarded as a coding theory problem as well. A $d$ dimensional $q$-ary covering code of radius $R=1$ is a subset $C$ of the vectors of dimension $d$ with coordinates in $\{0,1,\dots q-1\}$ such that for every vector there is a vector in $C$ which differs in at most $R$ coordinates from it. We claim that in the case $n=2$ it holds that $2^d-\ex_d(\cS_d,2)$ is equal to the maximal size of a $d$ dimensional binary covering code of radius $R=1$ (with respect to the Hamming distance), a well known open problem in coding theory. Indeed, the $0$-entries of a $d$-dimensional $q\times\dots\times q$ matrix $M$ that avoids $S_d$ gives a $d$ dimensional binary covering code of radius $R=1$ as every $1$-entry $p$ of $M$ has a $0$-entry $p'$ which differs in exactly one coordinate from $p$. In the other direction, having a binary covering code $C$ of radius $1$, let $M$ be a $0$-$1$ matrix whose $0$-entries are exactly those whose vectors are in $C$, then $M$ avoids all $d$-stars. Note that this holds only when $n=2$ as
for $n\ge 3$ having a $0$-entry $p'$ in some row of $p$ is not equivalent anymore to having no $1$-entry in this row besides $p$. Codes, including covering codes, have broad literature. For a list of results and citations about covering codes see, e.g., \cite{CHLL, Keri-10}.

\paragraph{Other problems about to $0$-$1$-matrices:}
An interesting connection of pattern avoidance problems is with forbidden subposet problems, see \cite{MePa2017}. A recent variant of the pattern avoidance problem that became popular is the saturation problem of multidimensional $0$-$1$ matrices~\cite{FuKe-21, Geneson-21, GeTs-23, Ber-23, Tsai-23}.

Yet another branch of $0$-$1$-matrix research is when we want to avoid induced copies of some pattern $A$. An induced copy of $A$ in $M$ is a submatrix of $M$ which is the same as $A$.  If we further allow permutation of rows and columns, we get the topic of forbidden subconfigurations where typically one wants to maximize the number of columns  when the number of rows is fixed such that the matrix avoids a given subconfiguration and no columns are repeated. For more on this see the survey of Anstee and Sali~\cite{AnSali-24}.

Further results about forbidden induced subgraphs consider equivalent characterizations and algorithmic questions. As an example we mention balanced $0$-$1$ matrices, which were introduced by Berge while studying the classes of perfect graphs \cite{Berge-72}. Related notions include totally balanced matrices and greedy matrices. Let us also mention the line of research regarding the existence of large homogeneous submatrices, see., e.g. the results of Kor\'andi, Pach, and Tomon~\cite{KoPa-20}.


\subsection{Open problems}

Note that for simplicity we mostly concentrated only on the case when all dimensions of the matrix $M$ are of the same size $n$. In the cases  $k=d=2$ and $k=d=3$, we proved exact results also for the case when the dimensions differ. It would be nice to have a similar exact bound for $k=2, d=3$:
\begin{conjecture}
	$\ex_3(\cS_2,n)=\frac{3}{2}n(n-1)+1$ if $n\ge 2$.
\end{conjecture}

\begin{problem}
	Further improve the bounds on the star avoidance problem. That is, given an integer $k \leq d$, find a maximum cardinality $S$, a set of $1$-entries,  in a $0$-$1$ matrix  $M$ such that $S$ contains no $k$-stars  where the dimension of $M$ is  $n_1\times n_2 \times \dots \times n_d$.
\end{problem}

\begin{problem}
Improve the bounds on the geodesic avoidance problem. That is, given an integer $k \leq d$, find a maximum cardinality $S$, a set of $1$-entries,  in a $0$-$1$ matrix  $M$ such that $S$ contains no dimension-distinct $k$-geodesics where the dimension of $M$ is  $n_1\times n_2 \times \dots \times n_d$.
\end{problem}

Both problems have the following variant: given a fixed subset $Q$ of the entries of $M$, find a maximal cardinality $S\subseteq Q$ with the appropriate property. This can be regarded as an algorithmic problem and also can be interesting to give bounds for special sets of $Q$.

\paragraph{Acknowledgment}
We are grateful to Bo\v stjan Bre\v sar and Sandi Klav\v zar for introducing this problem to us and for their insightful comments in the initial stages of our research.

\footnotesize
\bibliographystyle{plainurl}
\bibliography{01star.bib}

\end{document}